\documentclass[11pt]{article}
\usepackage{epsfig}
\usepackage{tikz}
\usepackage{caption}
\usepackage{subfig,graphicx}
\input{amssym}
\input{epsf}
\newtheorem{prethm}{{\bf Theorem}}

\newenvironment{theorem}{\begin{prethm}{\hspace{-0.5
em}{\bf.}}}{\end{prethm}}
\newtheorem{prelem}{{\bf Lemma}}

\newenvironment{Lemma}{\begin{prelem}{\hspace{-0.5
em}{\bf.}}}{\end{prelem}}
\newtheorem{cor}{{\bf Corollary}}

\newenvironment{Corollary}{\begin{cor}{\hspace{-0.5
em}{\bf.}}}{\end{cor}}
\newtheorem{prop}{{\bf Proposition}}

\newtheorem{preex}{{\bf Example}}

\newtheorem{presol}{{\bf Solution}}

\newtheorem{preproof}{{\bf Proof.}}

\newenvironment{proof}[1]{\begin{preproof}{\rm
               #1}\hfill{$\rule{2mm}{2mm}$}}{\end{preproof}}
\begin{document}

\date{}
\title{
{\Large{\bf Clique-coloring of $K_{3,3}$-minor free graphs }}}
{\small
\author{
{\sc  Behnaz Omoomi\footnote{bomoomi@cc.iut.ac.ir} } and {\sc Maryam Taleb  }\\
[1mm]
{\small \it  Department of Mathematical Sciences}\\
{\small \it  Isfahan University of Technology} \\
{\small \it 84156-83111, \ Isfahan, Iran}}




 \maketitle \baselineskip15truept

\begin{abstract}
A clique-coloring of a given graph $G$ is a coloring of the vertices of $G$ such that no maximal clique of size at least two is monocolored. The clique-chromatic number of $G$ is the least number of colors for which $G$ admits a clique-coloring. It has been proved  that every planar graph is $3$-clique  colorable and every claw-free planar graph, different from an odd cycle,  is $2$-clique colorable. In this paper, we generalize these results  to $K_{3,3}$-minor free ($K_{3,3}$-subdivision free) graphs. 
\end{abstract}


\section{Introduction} 
Graphs considered in this paper are all simple and undirected. Let $G$ be a graph with vertex set $V(G)$ and edge set $E(G)$. The number of vertices of $G$ is called the order of $G$.  The set of vertices adjacent  to a vertex $v$ is denoted by $N_G(v)$, and the size of $N_G(v)$  is called the degree of $v$ and is denoted by  $d_G(v)$. A vetex with degree zero is called an isolated vertex. The maximum degree of $G$ is denoted by $\Delta(G)$. For a subset $S\subseteq V(G)$, the subgraph induced by $S$ is denoted by $G[S]$. An independent  set is a set of vertices in graph, that does not induce any edge and the size of  maximum independent  set in  $G$ is written  by $\alpha(G)$. 

As usual,  the complete bipartite graph with parts of cardinality $m$ and $n$ $(m, n\in \mathbf{N})$ is indicated by $K_{m,n}$.  The graph $K_{1,3}$ is  called a claw.
The complete graph with $n$ vertices $\{v_1, \dots, v_n\}$ is denoted by $K_n$ or $[v_1, \dots, v_n]$. The graph $\bar{G}$ is the complement of $G$ with the same vertex set as $G$, and $uv$ is an edge in $\bar{G}$ if and only if it is not an edge in $G$. The path and the cycle of order $n$ are denoted by $P_n$ and $C_n$, respectively. The lenght of a path and a cycle is the number of its edges. A path with end vertices $u$ and $v$ is denoted by $(u,v)$-path.

   Edge $e$ is called an edge cut in connected graph $G$ if  $G/\{e\}$ is disconnected. A block in $G$ is a maximal $2$-connected subgraph of $G$. A chord of a cycle $C$ is an edge not in $C$ whose end vertices lie in $C$. A hole is a chordless cycle of length greater than three. A hole is said to be odd if its length is odd, otherwise, it is said to be even.  Given a graph $F$, a graph $G$ is called $F$-free if $G$ does not contain any induced subgraph isomorphic with $F$. A graph $G$ is a $(F_1,\ldots, F_k)$-free graph if it is $F_i$-free  for all $i \in \lbrace 1,\ldots, k\rbrace $. A graph $G$ is claw-free (resp. triangle-free) if it does not contain $K_{1,3}$ (resp. $K_3$) as an induced subgraph.
     
  By  a subdivision of  an edge $e=uv$,  we mean replacing the edge $e$ with a $(u,v)$-path. 
  Any graph derived from  graph $F$ by a sequence of  subdivisions is called a subdivision of $F$ or an $F$-subdivision. 
The contraction of an edge $e$ with endpoints  $u$ and $v$ is the replacement of $u$ and $v$ with a  vertex such that edges incident to the new vertex are the edges that were incident with either  $u$ or $v$ except $e$;  the obtained graph is denoted by 
 $G \cdot e$. Graph $F$ is called a minor of $G$ ($G$ is called  $F$-minor graph) if $F$ can be obtained from $G$ by a sequence of vertex and edge deletions and edge contractions. Given a graph $F$,  graph $G$ is $F$-minor free if $F$ is not a minor of $G$. Obviously, any graph $G$ which contains an $F$-subdivision also has an $F$-minor. Thus an $F$-minor free graph is necessarily $F$-subdivision free, although in general the converse is not true. However, if $F$ is a graph of  the maximum degree at most three, any graph which has an $F$-minor also contains an $F$-subdivision. Thus,  a graph is $K_{3,3}$-minor free if and only if it is $K_{3,3}$-subdivision free. By the well-known Kuratowski's theorem a graph is planar if and only if it is $K_5$-minor free and $K_{3,3}$-minor free. 
  For further information on graph theory concepts and terminology we refer the reader to \cite{bondy}.
 
 A vertex $k$-coloring of $G$ is a function $c: V(G) \longrightarrow \lbrace 1, 2,\ldots, k \rbrace $ such that for every two adjacent vertices $u$ and $v$, $c(u)\neq c(v)$. The minimum integer $k$ for which $G$ has a vertex $k$-coloring is called the chromatic number of $G$ and is denoted by $\chi(G)$. 
 A hypergraph $\mathcal{H}$ is a pair $(V, \mathcal{E})$, where $V$ is the set of vertices of $\mathcal{H}$, and $\mathcal{E}$ is a family of non-empty subsets of $V$ called hyperedges of $\mathcal{H}$. A \textit{$k$-coloring} of $\mathcal{H}=(V, \mathcal{E})$ is a mapping $c : V \longrightarrow \lbrace 1, 2,\ldots, k \rbrace $ such that for all $e \in \mathcal{E}$,  where $\vert e \vert \geq 2$, there exist $u, v \in e$ with $c(u) \neq c(v)$. The \textit{chromatic number} of $\mathcal{H}$, $\chi(\mathcal{H})$,   is the smallest $k$ for which $\mathcal{H}$ has a $k$-coloring. Indeed, every graph is a hypergraph in which every hyperedge is of size two and a $k$-coloring of such hypergraph is a usual vertex $k$-coloring.

A clique of $G$ is a subset  of mutually adjacent vertices of $V(G)$. A clique is said to be maximal if it is not properly contained in any other clique of $G$. We call \textit{clique-hypergraph} of $G$,  the hypergraph $\mathcal{H}(G)=(V,\mathcal{E})$ with  the same vertices as $G$ whose  hyperedges are the maximal cliques of $G$ of cardinality at least two. A $k$-coloring of $\mathcal{H}(G)$ is also called a \textit{$k$-clique coloring} of $G$, and the chromatic number of $\mathcal{H}(G)$ is called the \textit{clique-chromatic number} of $G$, and is denoted by $\chi_c(G)$. In other words, a $k$-clique coloring of $G$ is a coloring of $V(G)$ such that no maximal clique in $G$ is monochromatic, and $\chi_{c}(G)= \chi(\mathcal{H}(G))$. A clique coloring of $\mathcal{H}(G)$ is \textit{strong} if no triangle of $G$ is monochromatic. 
A graph $G$ is \textit{hereditary $k$-clique colorable} if $G$ and all its induced subgraphs are $k$-clique colorable. The clique-hypergraph coloring problem was posed by Duffus et al. in~\cite{cid}. To see more results on this concept see~\cite{ci3},~\cite{bt},~\cite{cims},~\cite{cit},~\cite{wu}.

  Clearly, any vertex $k$-coloring of $G$ is a $k$-clique coloring, whence $\chi_c(G) \leq \chi(G)$.  It is shown that in general, clique coloring can be a very different problem from usual vertex coloring and $\chi_c(G)$ could be much smaller than $\chi(G)$~\cite{ci3}. On the other hand, if $G$ is triangle-free, then $\mathcal{H}(G) =G$, which implies $\chi_{c}(G) =\chi(G)$. Since the chromatic number of triangle-free graphs is known to be unbounded~\cite{cim}, we get that the same is true for the clique-chromatic number of triangle-free graphs.
 In addition, clique-chromatic number of claw-free graphs or even line graphs is not bounded.
 For instance for each constant $k$, there exists $N_{k} \in \mathbf{N}$ such that for each $n\geq N_k$, $\chi_c(L(K_n))\geq k+1$ that $L(K_n)$ is line graph of complete graph $K_n$ and is claw-free~\cite{ci3}.
  On the other hand, D{\'e}fossez   proved that a claw-free graph is hereditary $2$-clique colorable if and only if it is   odd-hole-free~\cite{cid9}.  
 That is why  recognizing the structure of graphs with bounded and unbounded clique-chromatic number could be an interesting problem.

 For planar graphs, Mohar and Skrekovski in~\cite{ci1} proved the following theorem.

\begin{theorem}~\rm\cite{ci1}\label{a1}
 Every planar graph is strongly $3$-clique colorable.
\end{theorem}

 Moreover, Shan et al. in~\cite{ci2} proved the following theorem.

 \begin{theorem}~\rm\cite{ci2}\label{a2}
Every claw-free planar graph, different from an odd cycle, is $2$-clique colorable.
\end{theorem}

  Shan and Kang generalized the result of Theorem~\ref{a1} to $K_5$-minor free graphs and the result of 
  Theorem~\ref{a2} to graphs which are claw-free and $K_5$-subdivision free~\cite{cisk} as follows.
  
\begin{theorem}~\rm\cite{cisk}
 Every $K_5$-minor free graph is strongly $3$-clique colorable.
\end{theorem}

\begin{theorem}~\rm\cite{cisk}
 Every graph which is claw-free and $K_5$-subdivision free, different from an odd cycle, is $2$-clique colorable.
\end{theorem}


In this paper, we generalize the result of Theorem~\ref{a1} to $K_{3,3}$-minor free graphs and the result of Theorem~\ref{a2} to claw-free and $K_{3,3}$-minor ($K_{3,3}$-subdivision) free graphs.

\section{Preliminaries}
In this section, we state the structure theorem of claw-free graphs that is proved by Chudnovsky and Seymour~\cite{cmp}. At first we need a number of definitions.

Two adjacent vertices $u, v$ of graph $G$ are called twins if they have the same neighbors in $G$, and if there are two such vertices, we say $G$ admits twins. For a vertex $v$ in $G$ and a set $X \subseteq V (G) \setminus \lbrace v\rbrace $, we say that $v$ is complete to $X$ or $X$-complete if $v$ is adjacent to every vertex in $X$; and that $v$ is anticomplete to $X$ or $X$-anticomplete if $v$ has no neighbor in $X$. For two disjoint subsets $A$ and $B$ of $V(G)$, we say that $A$ is complete, respectively anticomplete, to $B$, if every vertex in $A$ is complete, respectively anticomplete, to $B$. A vertex is called singular if the set of its non-neighbors induces a clique.

 Let $G$ be a graph and $A, B$ be disjoint subsets of $V(G)$, the pair $(A, B)$ is called homogeneous pair in $G$, if  for every vertex $v \in V(G)\setminus (A\cup B)$, $v$ is either $A$-complete or $A$-anticomplete and either $B$-complete or $B$-anticomplete. If one of the subsets $A$ or $B$, for instance $B$ is empty, then $A$ is called a homogeneous set.

 Let $(A, B)$ be a homogeneous pair, such that $A, B$ are both cliques, and $A$ is neither complete nor anticomplete to $B$, and at least one of $A, B$ has at least two members. In these conditions the pair $(A, B)$ is called a $W$-join. 
  A homogeneous pair $(A, B)$ is non-dominating if some vertex of $V(G)\setminus (A\cup B)$ has no neighbor in $A\cup B$, and it is coherent if the set of all $(A\cup B)$-complete vertices in $V(G)\setminus (A\cup B)$ is a clique.

Next, suppose that $V_1, V_2$ is a partition of $V(G)$ such that $V_1, V_2$ are non-empty and $V_1$ is anticomplete to $V_2$. The pair $(V_1, V_2)$ is called a $0$-join in $G$.

Next, suppose that $V_1, V_2$ is a partition of $V(G)$, and for $i = 1, 2$ there is a subset $A_i \subseteq V_i$ such that:
\begin{itemize}
\item[1)]  $A_i$ is a clique, and $A_i$, $V_i \setminus A_i$ are both non-empty;
\item[2)] $A_1$ is complete to $A_2$;
\item[3)] $V_1 \setminus A_1$ is anticomplete to $V_2$, and $V_2 \setminus A_2$ is anticomplete to $V_1$.
\end{itemize}
 In these conditions, the pair $(V_1, V_2)$ is a $1$-join.

Now, suppose that $V_{0}, V_{1}, V_{2}$ is a partition of $V(G)$, and for $i = 1, 2$ there are subsets $A_i$, $B_i$ of $V_i$ satisfying the following properties:
\begin{itemize}
\item[1)]  $A_i$, $B_i$ are cliques, $A_i \cap B_i = \emptyset $, and $A_i$, $B_i$ and $V_i \setminus (A_i \cup B_{i})$ are all non-empty; 
\item[2)] $A_1$ is complete to $A_2$, and $B_1$ is complete to $B_2$, and there are no other edges between $V_1$ and $V_2$;
  \item[3)] $V_0$ is a clique, and, for $i = 1, 2$, $V_0$ is complete to $A_i\cup B_i$ and anticomplete to $V_i \setminus (A_i \cup B_i )$.
\end{itemize}
   The triple $(V_0, V_1, V_2)$ is called a generalized $2$-join, and, if $V_0 = \emptyset $, the pair $(V_1, V_2)$ is called a $2$-join.

The last decomposition is the following. Let $(V_1, V_2)$ be a partition of $V(G)$, such that for $i = 1, 2$, there are cliques $A_i , B_i ,C_i \subseteq V_i$ with the following properties:
\begin{itemize}
 \item[1)] the sets $A_i ,B_i , C_i$ are pairwise disjoint and have union $V_i$;
 \item[2)] $V_1$ is complete to $V_2$ except that there are no edges between $A_1$ and $A_2$, between $B_1$ and $B_2$, and between $C_1$ and $C_2$; and
  \item[3)] $V_1$, $V_2$ are both non-empty.
\end{itemize}  
   In these conditions it is said that $G$ is a hex-join of $V_1$ and $V_2$.

Now we define classes $F_0,\ldots ,F_7$ as follows:
\begin{itemize}
\item $F_0$ is the class of all line graphs.

\item The icosahedron is the unique planar graph with $12$ vertices of all degree five. For $k = 0, 1, 2, 3$, icosa($−k$) denotes the graph obtained from the icosahedron by deleting $k$ pairwise adjacent vertices. 
The class $ F_1$  is the family of all graphs $G$  isomorphic to icosa(0), icosa(−1), or icosa(−2).

\item Let $H$ be the graph with vertex set $\lbrace v_1,\ldots, v_{13} \rbrace $, with the following adjacency: $v_1v_2 \ldots v_6v_1$ is a hole in $G$ of length 6; $v_7$ is adjacent to $v_1, v_2$; $v_8$ is adjacent to $v_4, v_5$ and possibly to $v_7$; $v_9$ is adjacent to $v_6, v_1, v_2, v_3$; $v_{10}$ is adjacent to $v_3, v_4, v_5, v_6, v_9$; $v_{11}$ is adjacent to $v_3, v_4, v_6, v_1, v_9, v_{10}$; $v_{12}$ is adjacent to $v_2, v_3, v_5, v_6, v_9, v_{10}$;
$v_{13}$ is adjacent to $v_1, v_2, v_4, v_5, v_7, v_8$ and no other pairs are adjacent. 
The class  $F_2$ is the family of all graphs $G$  isomorphic to $H\setminus X$, where $X \subseteq \lbrace v_{11}, v_{12}, v_{13}\rbrace $.

\item Let $C$ be a circle, and $V(G)$ be a finite set of points of $C$. Take a set of subset of $C$ homeomorphic to interval [0, 1] such that there are not three intervals covering $C$ and no two  intervals share an end-point. Say that $u, v \in V(G)$ are adjacent in $G$ if the set of points $\lbrace u, v\rbrace $ of $C$ is a subset of one of the intervals. Such a graph is called circular interval graph. The class $F_3$ is the family  of all circular interval graphs.

\item Let $H$ be the graph with seven vertices $h_0,\ldots, h_6$, in which $h_1,\ldots, h_6$ are pairwise adjacent and $h_0$ is adjacent to $h_1$. Let $H^{'}$ be the graph obtained from the line graph $L(H)$  by adding one new vertex, adjacent precisely to the members of $V(L(H)) = E(H)$ that are not incident with $h_1$ in $H$. Then $H^{'}$ is claw-free. Let $F_4$ be the class of all graphs isomorphic to induced subgraphs of $H^{'}$. Note that the vertices of $H^{'}$ corresponding to the members of $E(H)$ that are incident with $h_1$ in $H$ form a clique in $H^{'}$. So the class  $F_4$ is the family of graphs that is  either a line graph or  has a singular vertex.

\item Let $n \geq 0$. Let $A = \lbrace a_1,\ldots, a_n\rbrace $, $B = \lbrace b_1,\ldots, b_n\rbrace $, $C = \lbrace c_1,\ldots, c_n\rbrace $ be three cliques, pairwise disjoint. For $1 \leq i, j \leq n$, let $a_i , b_j$ be adjacent if and only if
$i = j$, and let $c_i$ be adjacent to $a_j, b_j$ if and only if $i \neq j$. Let $d_1, d_2, d_3, d_4, d_5$ be five more vertices, where $d_1$ is $(A \cup B \cup C)$-complete; $d_2$ is complete to $A \cup B \cup \lbrace d_1\rbrace $; $d_3$
is complete to $A \cup \lbrace d_2\rbrace $; $d_4$ is complete to $B \cup \lbrace d_2, d_3\rbrace $; $d_5$ is adjacent to $d_3, d_4$; and there are no more edges. Let the graph just constructed be $H$. A graph $G \in F_5$ if (for some $n$) $G$ is isomorphic to $H\setminus X$ for some $X \subseteq A \cup B \cup C$. Note that vertex $d_1$ is adjacent to all the vertices but the triangle formed by $d_3, d_4$ and $d_5$, so it is a singular vertex in $G$.

\item Let $n \geq 0$. Let $A = \lbrace a_0,\ldots, a_n\rbrace $, $B = \lbrace b_0,\ldots ,b_n\rbrace $,
 $C = \lbrace c_1,\ldots , c_n\rbrace $ be three cliques, pairwise disjoint. For $0 \leq i, j \leq n$, let $a_i , b_j$ be adjacent if and only if $i = j > 0$, and for $1 \leq i \leq n$ and $0 \leq j \leq n$ let $c_i$ be adjacent to $a_j$ , $b_j$ if and only if $i \neq j \neq 0$. Let the graph just constructed be $H$. A graph $G \in F_6$ if (for some $n$) $G$ is isomorphic to $H\setminus X$ for some $X \subseteq (A\setminus \lbrace a_0\rbrace ) \cup (B\setminus \lbrace b_0\rbrace ) \cup C$.
 \item A graph $G$ is prismatic, if for every triangle $T$ of $G$, every vertex of $G$ not in $T$ has a unique neighbor in $T$. A graph $G$ is antiprismatic if its complement  is prismatic. The class  $F_7$ is the family  of all antiprismatic graphs.
 \end{itemize}
 
The structure theorem in~\cite{cmp} is as follows:
\begin{theorem}~\rm\cite{cmp}\label{a3}
If $G$ is a claw-free graph, then either
\begin{itemize}
\item $G \in F_0 \cup \cdots \cup F_7$, or
\item $G$ admits either twins, a non-dominating $W$-join, a $0$-join, a $1$-join, a generalized $2$-join, or a hex-join.
\end{itemize}
\end{theorem}

\section{$K_{3,3}$-minor free graphs}
In this section,  we focus on  the clique chromatic number of $K_{3,3}$-minor free graphs. In particular, we prove that every $K_{3,3}$-minor free graph is strongly $3$-clique colorable. Moreover, it is $2$-clique colorable if it
 is claw-free and different from an odd cycle.

For this purpose, first we need the Wagner charactrization of $K_{3,3}$-minor free graphs~\cite{wagner}.
Let $G_1$ and $G_2$ be graphs with disjoint vertex-sets. Also, let $k \geq  0$ be an integer, and for $i =1 ,2$, 
  let $X _i\subseteq V(G _i)$ be a set of cardinality $k$ of pairwise adjacent vertices. For $i =1,2$,  let $G'_ i$ be obtained from $G_i$ by deleting a (possibly empty) set of edges with both ends in $X_i$. If  $f:X _1\longrightarrow X _2$ is a bijection, and  $G$ is the graph obtained from the union of $G'_1$ and $G'_2$ by identifying $x$ with $f(x)$ for all $x \in X_1$, then  we say that $G$ is a $k$-sum of $G_1$ and $G_2$.

\begin{theorem}~\rm\cite{ci4,wagner}
A graph is $K_{3,3}$-minor free if and only if it can be obtained from planar graphs and  complete graph $K_5$ by meens of $0$-, $1$-, $2$-sums.
\end{theorem}

In order to make the above characterization easier,  we  use the structural sequence for $K_{3,3}$-minor free graphs. In fact, 
graph $G$ is $K_{3,3}$-minor free if and only if there  exists a sequence $\mathcal{T}=T_1, T_2, \ldots, T_r$, in which for each $i$, $1\leq i\leq r$, $T_i$ is either a planar graph or isomorphic with  $K_5$,  such that $G_1=T_1$, and  for each $i$,  $2\leq i\leq r$, $G_i$ is obtained from disjoin union of  $G_{i-1}$ and  $T_i$, or by gluing $T_i$ to   $G_{i-1}$ on  one vertex or one edge  or two non-adjacent vertices and $G_r=G$. For a given  $K_{3,3}$-minor free $G$, the sequence $\mathcal{T}$ is called a Wagner sequence.

Also we need following lemma proposed in~\cite{ci1}. 
\begin{Lemma}~\rm\cite{ci1}\label{l11}
 Let $G$ be a connected plane graph  such that its outer cycle, $C$, is a triangle. If $\phi:V(C) \longrightarrow \lbrace 1, 2, 3\rbrace $ is  a clique coloring of induced subgraph $C$, then $\phi$ can be extended to a strong $3$-clique coloring of $G$.
\end{Lemma}

In the following, we use the Wagner sequence  to provide a strong $3$-clique coloring for $K_{3,3}$-minor free graphs.
\begin{theorem}\label{a5}
Every $K_{3,3}$-minor free graph is strongly $3$-clique colorable.
\end{theorem}
\begin{proof}{
Let $G$ be a $K_{3,3}$-minor free graph. The assertion is trivial for $\vert V(G)\vert \leq 3$. So let $ \vert V(G) \vert \geq 4$ and $\mathcal{T} =T_1, T_2, \ldots, T_r$  be a  Wagner sequence of $G$. We use  induction on $r$. If $r=1$, then $G=T_1$ is either  $K_5$ or a planar graph. If $G$ is $K_5$, then the assertion is obvious, since by assigning  color $1$ to two vertices of $K_5$ and  color $2$ to two vertices of $K_5$ and  color $3$ to rest vertex, we have a strong $3$-clique coloring of $K_5$. Also, if $G$ is a planar graph, then the assertion follows directly from Theorem~\ref{a1}. 

Now let $r \geq 2$. By the induction hypothesis $G_{r-1}$ and $T_{r}$  have  strong $3$-clique coloring. If $G_r$ is $0$-sum of $G_{r-1}$ and $T_r$, then there is nothing to say.  Suppose that $G_r$ is obtained from $G_{r-1}$ and $T_r$ by gluing on vertex $\lbrace v\rbrace $. Thus, by a renaming  of the colors, if it is necessary, we obtain a strong $3$-clique coloring for $G_r$.

Next, we suppose that $G_r$ is obtained from $G_{r-1}$ and $T_r$ by gluing on edge $uv$ or two non-adjacent vertices $u$ and $v$. If $T_r$ is $K_5$, then we consider a strong $3$-clique coloring of $G_{r-1}$, say $\phi$,  and extend it to a strong $3$-clique coloring of $G_r$ as follow. 
If $\phi(u)\neq \phi(v)$, then we assign  three different colors $\lbrace 1,2,3\rbrace $ to the other three vertices of $K_5$. 
If $\phi(u)=\phi(v)$, then we assign  two different colors $\lbrace 1,2,3\rbrace \setminus \lbrace \phi(v)\rbrace $ to the other three vertices of $K_5$. Obviously, the extended coloring is a strong $3$-clique coloring of $G_r$.

Finally, let $T_r$ be a planar graph. We consider a strong $3$-clique coloring  of $G_{r-1}$, say $\phi$, and provide a strong $3$-clique coloring of $G_r$ as follow.  If $\phi(u)\neq \phi(v)$ and  $e=uv$ is a maximal clique of $T_r$, then suppose that $\phi^{'}$ is a  strong $3$-clique coloring of $T_r$. In this case,   by a renaming the color of $\phi'(u)$ and $\phi'(v)$ in $T_r$, if it is necessary, we obtain a strong $3$-clique coloring of $G_r$.  If $e=uv$ is not a  maximal clique in $T_r$, then there exists a triangle $T$ containing  $e$ in $T_r$.  Now we consider a planar embedding of $T_r$ in which $T$ is an outer face in it.  Hence,  by Lemma~\ref{l11}, it is enough to give a  strong $3$-clique coloring of outer cycle $T$ of plane graph $T_r$. 
That is obviously possible by coloring the third vertex of $T$ properly.

If $\phi(u)=\phi(v)$, then let $e=uv$ and  $T^{'}_{r}= T_{r} \cdot e$.  If there is no triangle consist of $e=uv$ in $T^{'}_r$, then we consider a strong $3$-clique coloring $\phi^{'}$ of plane graph $T^{'}_r$,  such that $\phi^{'}(u)=\phi^{'}(v)=\phi(u)=\phi(v)$. Note that edge $e=uv$ is not maximal clique in $G_{r-1}$,  so it is not maximal clique in $G_{r}$. Therefore,  the coloring $\phi(x)$ for $x \in G_{r-1}$ and $\phi^{'}(x)$ for $x \in T_{r} \cdot e$ is a strong $3$-clique coloring for $G_r$. 
 If $e=uv$ is  in triangle $T$ in $T^{'}_{r}$, then we consider a planar embedding of $T^{'}_r$ in which $T$ is an outer face in it. By Lemma~\ref{l11}, it is enough to give a $3$-clique coloring of outer cycle $T$ of plane graph $T^{'}_r$. Thus, we give  $\phi'(u=v)=\phi(u)=\phi(v)$ and   assign two different colors $\lbrace 1,2,3\rbrace \setminus \lbrace \phi(v)\rbrace $ to other two vertices of $T$, then  we extend $\phi'$ to a strong $3$-clique coloring  of $T^{'}_r$. This implies a strong $3$-clique coloring of $T_r$ as desired, and again we obtain a strong $3$-clique coloring of $G_r$.
}\end{proof}

The rest of this section deals with the proof that, 
every claw-free and $K_{3,3}$-minor free graph~$G$, different from an odd cycle of order greater than three, is $2$-clique colorable.
For this purpose,  we need  two following theorems.
 
\begin{theorem}~\rm\cite{cims}\label{aa}
If $G\in F_1\cup F_2\cup F_3\cup F_5\cup F_6$ or $G$ admits a hex-join, different from an odd cycle of order greater than three, then $G$ is $2$-clique colorable.
\end{theorem}

\begin{theorem}~\rm\cite{cims}\label{af}
Every connected claw-free graph $G$ with maximum degree at most seven, not an odd cycle of order greater than three, is $2$-clique colorable.
\end{theorem}

From the  proof of Theorem~\ref{af},  we conclude the following corollary.
\begin{Corollary} \label{cf}
If $G$ is a  connected $K_{3,3}$-minor free graph which  admits either  twins, or a non-dominating $W$-join, or a coherent $W$-join, or a $1$-join, or a generalized $2$-join, except an odd cycle of order greater than three, then $G$ is $2$-clique colorable.
\end{Corollary}

According to Theorem \ref{aa} and Corollary \ref{cf}, it is sufficient to show that every $K_{3,3}$-minor free  graph
 $G\in F_0\cup F_4\cup F_7$ except an odd cycle of order greater than three, is $2$-clique colorable.
  First we show this result for  class $F_0$ (the class of line graphs).

\begin{prop}\label{a7}
Every  $K_{3,3}$-minor free graph  in $F_0$, different from an odd cycle of order greater than three, is $2$-clique colorable.
\end{prop}
 
 \begin{proof}{
Let $G$ be a  $K_{3,3}$-minor free line graph. The assertion is trivial for $\vert V(G)\vert \leq 3$. Now, let $ \vert V(G) \vert \geq 4$.   Let $\mathcal{T} =T_1, T_2, \ldots, T_r$ be a Wagner sequence of $G$. We use induction on  $r$.
  If $r=1$, then $G=T_1$ is either  $K_5$ or a planar graph. If $G$ is $K_5$, then the assertion is obvious. If $G$ is a planar graph, then by Theorem \ref{a2},  $G$ has a $2$-clique coloring,  since every line graph is claw-free.
  
 Now let $r \geq 2$. By the induction hypothesis $G_{r-1}$ and $T_r$  have  $2$-clique coloring.
If $G_r$ is $0$-sum or $1$-sum of $G_{r-1}$ and $T_r$, then the result is obvious.
Now, we suppose that $G_r$ is $2$-sum of $G_{r-1}$ and $T_r$ on edge $uv$. 
Note that if $uv$ is an edge cut, then $G$ can be considered as $1$-sum of two graphs. So, later on we assume that $uv$ is not an edge cut.
If $T_r$ is $K_5$ and  $\phi$ is a $2$-clique coloring of  $G_{r-1}$,  then we assign the colors $\phi(u)$ and $\phi(v)$ to vertices $u, v$ in $K_5$ and give two different colors $\{1,2\}$ to the other three vertices of $K_5$. 

If  $T_r$ is a planar graph, then we have four possibilities: \\
{\rm $(i)$} there exists $2$-clique colorings $\phi$ and  $\phi^{'}$ of $G_{r-1}$ and $T_r$, such that  $\phi(u)\neq \phi(v)$ and $\phi^{'}(u)\neq \phi^{'}(v)$; \\
 {\rm  $(ii)$} there exists $2$-clique colorings $\phi$ and  $\phi^{'}$ of $G_{r-1}$ and $T_r$,  such that $\phi(u)=\phi(v)$ and $\phi^{'}(u)=\phi^{'}(v)$; \\
 {\rm $(iii)$} in every $2$-clique colorings $\phi$ and  $\phi^{'}$ of $G_{r-1}$ and $T_r$,  $\phi(u)\neq \phi(v)$ and $\phi^{'}(u)=\phi^{'}(v)$; \\
 {\rm $(iv)$}  in every $2$-clique colorings $\phi$ and  $\phi^{'}$ of $G_{r-1}$ and $T_r$,    $\phi(u)= \phi(v)$ and $\phi^{'}(u)\neq \phi^{'}(v)$.
 
In the first two cases,    only by a color renaming, if it is necessary, we obtain a $2$-clique coloring for $G_{r}$. In the following, without loss of generality we consider the case $(iii)$ and show that it is impossible. 

The assumption $(iii)$,   concludes  that  vertex $u$ (and $v$)  in $T_{r}$ belongs to a maximal clique  $C_u$ (and $C_v$) such that  in every $2$-clique coloring of $T_r$,  $C_u\setminus \{u\}$ (and $C_v\setminus \{v\}$) is monochromatic. Hence, $u\notin C_v$ and $v\notin C_u$. This implies that, $u$ has a non-neighbor vertex in $C_v$, say $v'$,  also $v$ has a non-neighbor vertex in $C_u$, say $u'$.
Moreover,  assumption $(iii)$ implies that $uv$ is a maximal clique in $G_{r-1}$. Thus, 
  there exist vertex $u'' \in N_{G_{r-1}}(u)$  that  $u''\notin N_{G_{r-1}}(v)$ (or $v'' \in N_{G_{r-1}}(v)$  that  $v''\notin N_{G_{r-1}}(u)$). Hence,   edge  $uv$ among  edges $uu'$  and $uu''$  (or $vv'$ and  $vv''$) 
  is  a claw in $G_{r}$, that is a contradiction.

If  in the  operation $2$-sum,  the edge $uv$ is deleted, then by the following argument,  we could change the coloring of vertices in $T_r$ such that $\phi^{'}(u)\neq \phi^{'}(v)$,   that contradicts the assumption $(iii)$. Note that since $uv$ is not an edge cut in $G_{r-1}$ and $T_r$, there are shortest  $(u,v)$-paths $P: u_0=uu_1\dots u_s=v$ in $T_r/\{uv\}$ and $Q: v_0=vv_1\dots v_t=u$  in $G_{r-1}/\{uv\}$. 
  Since $G_r$  is claw-free, vertices $u$ and $v$ in $T_{r}$ and $G_{r-1}$ belongs to only one maximal clique.  
If  $d_{T_{r}}(u_i)=2$, $i=1, \dots, s-1$ and $d_{G_{r-1}}(v_j)=2$, $j=1, \dots, t-1$, then by $(iii)$, the lenght of $P$ is even and 
the lenght of $Q$ is odd. This implies  $G_r$ is an odd cycle and contradicts our assumption.
Thus, assume that $k\in \{0,1,\dots, s-1\}$ is the smallest indices that $d_{T_{r}}(u_k)\geq 3$ and  $w\in N_{T_{r}}(u_k)$. 
Since $G_r$ is claw free, we  must have $w\in N_{T_{r}}(u_{k+1})$. 
Let $C$ be a unique maximal clique consist of  $[u_k, u_{k+1}, w]$ (note that $N_{T_r}(u_k) \subseteq N_{T_r}(u_{k+1})$). If there exists a vertex in $C$ that its color is  
$\phi^{'}(u_k)$, then  we  swap the colors of vertices on  $(u,u_k)$-path in $P$. Thus, we will obtain a $2$-clique coloring of $T_r$ such that $u$ and $v$  are assigned different colors. This contradicts the assumption $(iii)$. 

Now assume that  the color of all vertices in $C$ are different from $\phi^{'}(u_k)$.  In this case, 
if  there exists a vertex in $C$, say $w'\neq u_k$, such that $C$ is a unique maximal clique contains $w'$, then we assign $\phi^{'}(u_k)$ to $w'$ and again swap the colors of vertices on  $(u,u_k)$-path in $P$. 
Otherwise, every vertex in $C$  belongs to a maximal clique other than $C$.  In this case, if there exists a vertex  $w'\in C$,  such that 
$w'\in C'$, where $C$ and $C'$ are maximal cliques in different blocks of $T_r$, then we swap the color of vertices in the component of $T_r/\{w'\}$ consist of $C'$, assign $\phi^{'}(u_k)$ to $w'$ and again swap the colors of vertices on  $(u,u_k)$-path in $P$.
Thus, we will obtain a $2$-clique coloring of $T_r$ such that $u$ and $v$  are assigned different colors. This contradicts the assumption $(iii)$.

The remaining case is that all vertices in $C$ belong to some other maximal cliques and all cliques are in one block in $T_r$.
In this case, let $l$ is the smallest indices that there exists  a path from $u_l$ to some vertices in $C/\{u_k,u_{k+1}\}$, we call this path $(w, u_l)$-path $P': ww_1\dots w_m=u_l$. 
Note that if there is not such a path, then we can consider graph $G$  as a $2$-sum of two graphs on edge $u_ku_{k+1}$, and we are  done.
If $m=1$, then since $P$ is a shortest path, we have $l=k+2$. Therefore, the induced  subgraph on vertices $\{u_{k-1}, u_k, u_{k+1}, u_{k+2}, u_{k+3}, w\}$ is  one of the nine forbidden structures in line graphs (see~\cite{beineke1968})  (note that if $k=0$ or $k=s-2$, then vertex $u_{k-1}=v_{t-1}$  or  $u_{k+3}=v_1$).  Hence, $m\geq 2$.  Also, $w_{m-1}$ is adjacent to $u_{l+1}$, since $T_r$ is claw free.
Now, by considering the  first internal  vertices  in  $P'$ and $(u_{k+1}, u_l)$-path in $P$ with degree greater than two, we do the similar above discussion in order to  change the  color of vertices $w$ or $u_{k+1}$ and subsequently change the color of $u$. 
Therefore, if we could not do that,  then we conclude that pattern of  colors in these paths are $a,b,a,b \dots$, where $a, b\in \{1,2\}$. Now, we have $\phi^{'}(w_{m-1})= \phi^{'}(u_{l+1})\neq  \phi^{'}(u_{l})$ or 
$\phi^{'}(w_{m-1})\neq \phi^{'}(u_{l+1})$. 
 In the former case, we swap the color of vertices in path $w_{m-1}w_{m-2} \dots w_1wu_k\dots u$. In the latter case, we swap the color of vertices in path $u_l u_{l-1}\dots u_{k+1}u_k u_{k-1}\dots u$. Thus, in both cases, we obtain a $2$-clique coloring for $T_r$ such that the vertices $u$ and $v$ recieve different colors and this contradicts the assumption $(iii)$. Therefore, the cases $(iii)$ and $(iv)$ are impossible and the proof is complete.
}\end{proof}

Now we show the $2$-clique colorability  of $K_{3,3}$-minor free   graphs in class  $F_4$.  First, we need the following theorem. 
\begin{theorem}~\rm\cite{ci3}\label{a8}
For any graph $G\neq C_5$ with $\alpha(G) \geq 2$, we have $\chi_c(G) \leq \alpha(G)$.
\end{theorem}
\begin{prop}\label{a9}
Every  $K_{3,3}$-minor free graph in $F_4$ is $2$-clique colorable.
\end{prop}
\begin{proof}{
Let $G$ be a graph in $F_4$. Since a graph in $F_4$ is a line graph or has a singular vertex, by Proposition~\ref{a7} it is sufficient to consider graphs in $F_4$ with singular vertex. So by the constraction of graphs in $F_4$, we have $\alpha(G) \leq 3$. For case $\alpha(G)=1$, the statement is obvious. If $\alpha(G)=2$, then by Theorem~\ref{a8}, $G$ is $2$-clique colorable, otherwise $\alpha(G)=3$. Let $x$ be a  singular vertex and $S=\lbrace  r, s, t\rbrace $ be  a maximum independent set in  $G$. Note that, $x\notin S$ and since non-neighbor vertices of $x$ induce  a clique, vertices $ r, s$ are adjacent to $x$ and $t$ is not adjacent to $x$.

 Now we propose a $2$-clique coloring $\phi$ for $G$ as follow:  let $\phi(x)=1$, $\phi(t)=2$ and   assign color $1$ to every non-neighbor vertex of $x$ except $t$. Now if $x$ and $t$ have more than one common neighbor,  then assign color $2$ to one of them and color $1$ to the other vertices, otherwise assign color $1$ to their common neighbor. Finally,  assign color $2$ to the other adjacent vertices to $x$. It  is easy to  see that, this assignment is a $2$-clique coloring of $G$.
}\end{proof}

Finally,  we show the $2$-clique colorability  of  $K_{3,3}$-minor free graphs  in  class $F_7$.
\begin{prop}\label{a10}
 Every  $K_{3,3}$-minor free graph in $F_7$ is  $2$-clique colorable.
\end{prop}
\begin{proof}{
Let $G$ be a graph in $F_7$. Since $G$ is an antiprismatic, $\bar{G}$ is  prismatic. If $\bar{G}$ has no triangle, then $\alpha(G)=2$, and  by Theorem~\ref{a8},  is $2$-clique colorable. Now let $T= [vuw] $ be a triangle in $\bar{G}$, and $S_1 = N_{\bar{G}}(v)\setminus \lbrace u,w\rbrace $, $S_2 = N_{\bar{G}}(u)\setminus \lbrace v,w\rbrace $ and $S_3 = N_{\bar{G}}(w)\setminus \lbrace u, v\rbrace $ be a partition of vertices $V(G) -\lbrace v, u, w\rbrace $.

Liang et al. in~\cite{cims} prove that if
\begin{itemize}
\item[(i)]  $\vert S_i\vert = 0$ for some $i = 1, 2, 3$, then $G$ has a $2$-clique coloring.
\item[(ii)] $\vert S_i\vert = 1$ for some $i = 1, 2, 3$, then $G$ has a $2$-clique coloring.
\item[(iii)] there is an edge $xy$ in $\bar{G}$ such that for $i\neq j\in \lbrace 1,2,3\rbrace $, $x$ is an isolated vertex in $\bar{G}[S_i]$ and $y$ is an isolated vertex in $\bar{G}[S_j]$, then there exists a $2$-clique coloring of $G$.
\item[(iv)] there exist $i\neq j\in \lbrace 1,2,3\rbrace $ such that $S_i\cup S_j$ is an independent set in $\bar{G}$, then $G$ has a $2$-clique coloring.
\end{itemize}
In the following for the remaining cases,  we provide a $2$-clique coloring for $G$ or we show that $G$ is $K_{3,3}$-minor that is a contradiction. Let $S_1 = \lbrace v_1, v_2\rbrace $ and  $S_2 = \lbrace u_1, u_2\rbrace $ and  $S_3 = \lbrace w_1, w_2\rbrace $. There are $i\neq j$, $i,j\in \lbrace 1,2,3\rbrace $, say $i=1$, $j=2$, such that $v_1$ is adjacent to $v_2$  in $\bar{G}$ and  $u_1$ is adjacent to $u_2$ in $\bar{G}$, otherwise by (iii) or (iv), we have $\chi_c(G)\leq 2$. Hence,  we have triangles $[uu_1u_2]$ and $[vv_1v_2]$ in $\bar{G}$. Since $\bar{G}$ is a prismatic $v_1,v_2,w_1,w_2$ have a unique neighbor in $[uu_1u_2]$ and $u_1, u_2,w_1, w_2$ have a unique neighbor in $[vv_1v_2]$. Thus, $\lbrace u_1, u_2,v_1,v_2\rbrace $ induces a cycle in $\bar{G}$ because, otherwise, for instance if $u_1$ and $u_2$ both are adjacent to $v_1$, then there exist two neighbors for $u$ in triangle $[u_1u_2v_1]$. Without loss of generality, assume that $u_1v_1$ and $u_2v_2$ are edges in $\bar{G}$. That means,  $u_1v_2$ and $u_2v_1$ are edges in $G$.

 Now each two vertices $w_1$ and $w_2$ have unique neighbor in $[u u_1 u_2]$ and $[v v_1 v_2]$. If both vertices $w_1$ and $w_2$ are adjacent to $u_1$ (or $u_2$) and $v_1$ (or $v_2$) in $\bar{G}$, then there exists two neighbors for $w_2$ in triangle $[v_1u_1w_1]$ (or $[v_2u_2w_1]$) that  contradicts $\bar{G}$ is prismatic. If vertices $w_1$ and $w_2$ are  both  adjacent to $u_1$ (or $u_2$) and $v_2$ (or $v_1$) in $\bar{G}$,  then  $G$ has a $K_{3,3}$-minor, on vertices $\lbrace w, w_1, w_2 ; u, v, v_1\rbrace $ (or $\lbrace w, w_1, w_2; u, v, v_2\rbrace $).  Note that if  $w_1$ is adjacent to $w_2$ in $\bar{G}$, then  we have triangle $[ww_1w_2]$ and since $\bar{G}$ is  prismatic,  vertices $w_1$ and $w_2$ cannot be both adjacent to one vertex of $\lbrace v_1, v_2\rbrace $ or $\lbrace u_1, u_2\rbrace $.  If $w_1$ is adjacent to $u_1$ (or $u_2$)  and $v_1$ (or $v_2$) and $w_2$ is adjacent to $u_2$ (or $u_1$)  and $v_2$ (or $v_2$)  in $\bar{G}$, then $G$ has  a $K_{3,3}$-minor, on vertices $\lbrace w, w_1, w_2; u, v, v_2\rbrace $ (or $\lbrace w, w_1, w_2; u, v, v_1\rbrace $). Hence,  all cases above,   contradicts that $G$  is $K_{3,3}$-minor free or  $\bar{G}$ is  prismatic. Thus, it is enough to consider the two following remaining cases.
 \begin{itemize}
 \item $w_1$ is adjacent to $u_1$ and $v_2$, and $w_2$ is adjacent to $u_2$ and $v_1$ in $\bar{G}$ (Figure (\ref{G2}) shows graph $G$).
  \item $w_1$ is adjacent to $u_2$  and $v_1$, and $w_2$ is adjacent to $u_1$ and $v_2$  in $\bar{G}$ (Figure (\ref{G1}) shows graph $G$).
 \end{itemize}

 In both above cases $G$ is a  claw free planar graph and  by Theorem~\ref{a2} is $2$-clique colorable (in Figure 1, the dashed lines show the edges that may exist or not exist in $G$). 

 \begin{figure}[ht] 
 \centering \subfloat[ ]
  { \includegraphics[scale=.26]{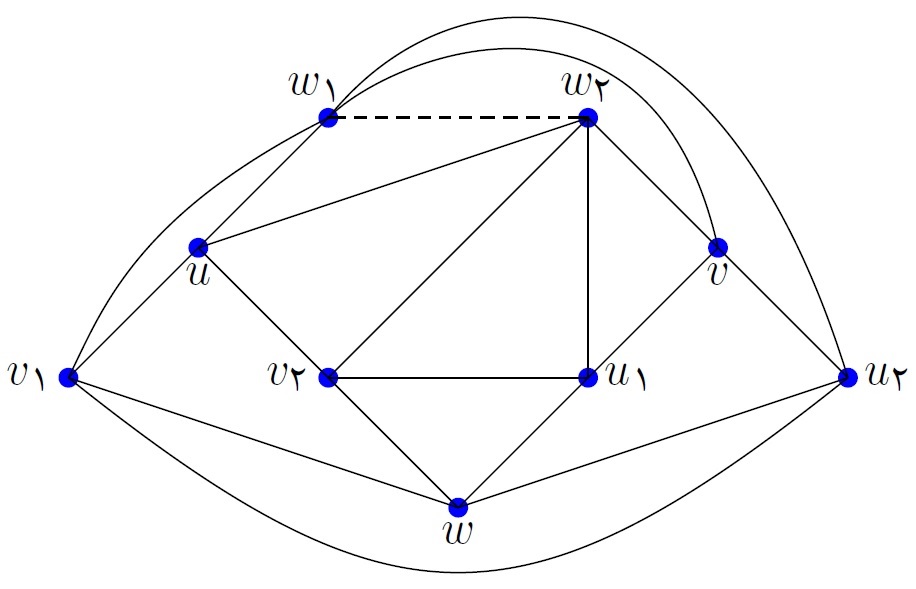}\label{G2}}
  \qquad
  \subfloat[ ]
 {\includegraphics[scale=.26]{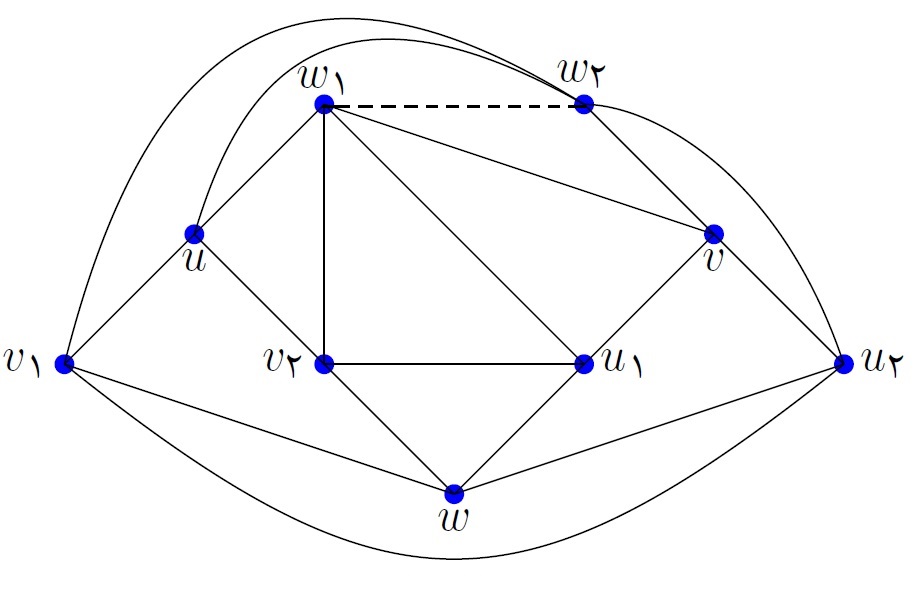}\label{G1}}
   \caption{Two $K_{3,3}$-minor free graphs.}\label{fig} 
   \end{figure} 

Finally let $\vert S_i\vert \geq 3$ for some $i = 1, 2, 3$,  say    $\vert S_1\vert \geq 2$,  $\vert S_2\vert \geq 2$  and $S_3 = \lbrace w_1, w_2, w_3\rbrace $.  Since such graphs contains the graphs with $|S_i| \leq 2$, $i=1, 2, 3$ as subgraph, we only need to consider graphs that contains one of the two graphs shown in Figure~1. 
By case (iv) there are $i\neq j \in \lbrace 1, 2, 3\rbrace $ such that $\bar{G}[S_i]$ and $\bar{G}[S_j]$ both are not independent. Liang et al. in~\cite{cims} show that $\bar{G}[S_i]$, $i\in \lbrace1, 2, 3\rbrace $, is not path and triangle. So we need to consider the case that 
 $[u u_1 u_2]$ and  $[v v_1 v_2]$    are triangles in $\bar{G}$, and $v_1w_3\in E(\bar{G})$ or $v_2w_3\in E(\bar{G})$. This implies    $G$ has a $K_{3,3}$-minor, on vertex set 
$\lbrace w, w_1, w_2; u, v, v_2\rbrace $ or  $\lbrace w, w_1, w_2; u, v, v_1\rbrace $, respectively. Note that, when
 $[u u_1 u_2]$ and  $[w w_1 w_2]$  are triangles in  $\bar{G}$,  is similar.
Therefore,  when $\vert S_i\vert \geq 3$ for some $i = 1, 2, 3$, $G$ is  a $K_{3,3}$-minor, that is a contradiction.
}\end{proof}

By Theorem~\ref{aa}, Corollary~\ref{cf} and Proposition~\ref{a7}, \ref{a9}, \ref{a10}, the main result in this section is proved.

\begin{theorem}\label{m}
If $G$ is  claw-free and $K_{3,3}$-minor free graph  except an odd cycle of order greater than three, then $G$ is $2$-clique colorable.
\end{theorem}

\end{document}